\documentclass[a4paper, 11pt]{article}

\newcommand{\myTitle}{Complete Solution of the Lady in the Lake Scenario}
\usepackage{amsmath, amssymb, amsthm}
\usepackage{marginnote}
\usepackage[colorlinks=true,pagebackref=false]{hyperref}
\usepackage{xcolor}
\hypersetup{urlcolor=purple, citecolor=red, linkcolor=blue, pdftitle={\myTitle}, pdfauthor={Von Moll, Pachter}}
\usepackage[style=ieee, sorting=none]{biblatex}
\usepackage{csquotes}
\usepackage[font=small, labelfont=bf]{caption}
\usepackage{mathrsfs}
\usepackage{alphabeta}
\usepackage{csquotes}

\newtheorem{theorem}{Theorem}
\newtheorem{proposition}{Proposition}
\newtheorem{lemma}{Lemma}
\theoremstyle{remark}
\newtheorem{remark}{Remark}
   
\usepackage{cleveref}
\crefmultiformat{equation}{(#2#1#3)}%
{ and~(#2#1#3)}{, (#2#1#3)}{ and~(#2#1#3)}
\crefrangeformat{equation}{(#3#1#4) to~(#5#2#6)}
\crefformat{equation}{(#2#1#3)}
\crefmultiformat{figure}{Fig.~#2#1#3}%
{ and~Fig.~#2#1#3}{, Fig.~#2#1#3}{ and~Fig.~#2#1#3}
\crefrangeformat{figure}{Figs.~#3#1#4 to~#5#2#6}
\crefformat{figure}{Fig.~#2#1#3}

\DeclareMathOperator*{\argmin}{\arg\min}
\DeclareMathOperator*{\argmax}{\arg\max}
\DeclareMathOperator{\sign}{sign}
\renewcommand{\H}{\mathscr{H}}

\usepackage{tikz}
\usepackage{tikzscale}
\usetikzlibrary{arrows.meta, arrows, math, angles, quotes, shapes.geometric}
\usepackage{pgfplots}
\pgfplotsset{compat=newest}
\usetikzlibrary{backgrounds}
\usepgfplotslibrary{patchplots}
\usepgfplotslibrary{fillbetween}
\pgfplotsset{%
    layers/standard/.define layer set={%
        background,axis background,axis grid,axis ticks,axis lines,axis tick labels,pre main,main,axis descriptions,axis foreground%
    }{
        grid style={/pgfplots/on layer=axis grid},%
        tick style={/pgfplots/on layer=axis ticks},%
        axis line style={/pgfplots/on layer=axis lines},%
        label style={/pgfplots/on layer=axis descriptions},%
        legend style={/pgfplots/on layer=axis descriptions},%
        title style={/pgfplots/on layer=axis descriptions},%
        colorbar style={/pgfplots/on layer=axis descriptions},%
        ticklabel style={/pgfplots/on layer=axis tick labels},%
        axis background@ style={/pgfplots/on layer=axis background},%
        3d box foreground style={/pgfplots/on layer=axis foreground},%
    },
}
\tikzset{point/.style={circle, fill, inner sep=1.7}}
\tikzset{point2/.style={regular polygon, regular polygon sides=3, fill, inner sep=1.3}}
\usepackage{subfig}

\title{\myTitle
    \footnote{
    This paper is based on work performed at the Air Force Research Laboratory (AFRL) \textit{Control Science Center}.
    DISTRIBUTION STATEMENT A.
    Approved for public release.
    Distribution is unlimited.
    AFRL-2024-0127; Cleared 09 JAN 2024.
}}
\author{Alexander Von Moll \and Meir Pachter}
\date{December 2023}

\begin{document}
\maketitle 

\begin{abstract}
    In the Lady in the Lake scenario, a mobile agent, $L$, is pitted against an agent, $M$, who is constrained to move along the perimeter of a circle.
    $L$ is assumed to begin inside the circle and wishes to escape to the perimeter with some finite angular separation from $M$ at the perimeter.
    This scenario has, in the past, been formulated as a zero-sum differential game wherein $L$ seeks to maximize terminal separation and $M$ seeks to minimize it.
    Its solution is well-known.
    However, there is a large portion of the state space for which the canonical solution does not yield a unique equilibrium strategy.
    This paper provides such a unique strategy by solving an auxiliary zero-sum differential game.
    In the auxiliary differential game, $L$ seeks to reach a point opposite of $M$ at a radius for which their maximum angular speeds are equal (i.e., the antipodal point).
    $L$ wishes to minimize the time to reach this point while $M$ wishes to maximize it.
    The solution of the auxiliary differential game is comprised of a Focal Line, a Universal Line, and their tributaries.
    The Focal Line tributaries' equilibrium strategy for $L$ is semi-analytic, while the Universal Line tributaries' equilibrium strategy is obtained in closed form.
\end{abstract}


\section{Introduction}
\label{sec:Introduction}

The Lady in the Lake scenario involves a mobile agent, the Lady, denoted $L$, swimming in a circular lake and another agent, the Man (or Monster), denoted $M$, whose motion is constrained to the perimeter of the lake.
$L$ seeks to reach the perimeter with maximum angular separation from $M$ while the latter seeks to minimize the angular separation.
$L$'s swimming speed is less than $M$'s running speed (otherwise the solution is relatively trivial), however, upon reaching the shore, $L$ can run faster than $M$.

The scenario first appeared in a column in \textit{Scientific American} by Martin Gardner in 1965.
This original problem description was later collected in a book~\cite{gardner1975mathematical} and was also posted in a collection of Gardner's writings~\cite{gardner2008lady}.
Later, the scenario was formulated as a zero-sum differential game and solved as an example in~\cite{breakwell1977lecture}.
Again, the scenario was included as an example in Ba{\c s}ar and Olsder's book~\cite{basar1982chapter} and an analytical solution was provided therein.
According to~\cite{basar1982chapter} the scenario also appeared in the Russian translation of Isaacs' book~\cite{isaacs1965differential}.
Then the scenario was revisited in~\cite{falcone2006numerical}, although, instead of analysis and geometry, numerical methods were used to approximate a solution (presumably because an analytical solution already existed for comparison purposes).
These numerical methods were based upon viscosity solutions of the Hamilton-Jacobi-Isaacs (HJI) partial differential equation.

More recently, the Lady in the Lake scenario has been reintroduced, in much the same way as the original, in the magazine \textit{Quanta} as a mathematical puzzle~\cite{mutalik2021math,mutalik2021mathcan}.
However, in its new incarnation, a twist has been added: in~\cite[Puzzle 2]{mutalik2021math} the reader is asked to determine (essentially) the equilibrium escape time which $L$ seeks to minimize and $M$ seeks to maximize when $L$ starts in the center of the lake.
The readers' and author's solutions account for the possibility of $M$ changing direction in order to foil $L$'s strategy in an effort to drive at the true equilibrium solution.
However, a full differential game treatment of this problem (as well as the more general scenario of any starting position for $L$) has not yet been presented and is out of the scope of the current paper.
Nonetheless, \cite{vonmoll2022circular} analyzed an easier variant of this problem for which $L$ begins \textit{outside} the lake and seeks to enter in minimum time subject to keeping $θ$, the angle between $L$ and $M$, non-zero.
There are also several papers on the topic of evading a finite-range Turret whose solutions resemble the original Lady in the Lake solution with a few added subsolutions~\cite{ivanov1993problem,galyaev2013evading,vonmoll2023turret}.

Although the solution to the original Lady in the Lake scenario have been well-established, several open questions remain (and were mentioned in~\cite{basar1982chapter}).
These questions have to do with a particular point in the lake from which $L$ can guarantee its minimum terminal angular separation.
This point is opposite of $M$ at a radius from the lake's center for which $L$ and $M$'s maximum angular speeds are equal, henceforth, the \textit{antipodal point}, or $E$.
If $L$ were to start under the equilibrium trajectory emanating from $E$ she'd do best by first reaching $E$ and subsequently exiting the lake along the associated equilibrium trajectory.
Therefore, the following questions was raised:
\begin{displaycquote}[p.\ 394,][paraphrased]{basar1982chapter}[]
    If $L$ starts at the lake center and knows $M$'s current action, show that $L$ will reach the antipodal point, $E$.
\end{displaycquote}
Some natural extensions to this question then include:
\begin{enumerate}
    \item How long will it take for $L$ to reach $E$ (i.e., what is the equilibrium, $\min \max$ time)?
    \item What if $L$ starts from general position (i.e., not just starting at the center of the lake)?
\end{enumerate}
This paper answers all of these questions, effectively completing the solution of the Lady in the Lake differential game by providing a unique strategy for the players in a large region of the state space for which the canonical strategy is undefined.

A recent work~\cite{wang2022solving} has sought to address very similar questions.
There, the focus is on finding the minimum time trajectory for $L$ in the region of the state space where she has angular speed advantage over $M$ (which is only a subset of the region for which the canonical strategy for the $\min \max$ terminal angle game is non-unique/undefined).
Ultimately, the authors specify a nonlinear program which utilizes a general numerical optimization solver to obtain minimum time trajectory resulting in $L$ maneuvering to $E$.
This paper builds upon that work by providing a solution which is closed-form for part of the state space and semi-analytic in the other part.

Following in the footsteps of~\cite{isaacs1965differential,breakwell1977lecture,basar1982chapter}, the methodology used within this paper is based upon differential game theory.
In general, obtaining solutions to differential games is a difficult endeavour as it involves solving the HJI, a technique which suffers from the curse of dimensionality~\cite{bernhard2014pursuit-evasion}.
For example, the Homicidal Chauffeur Differential Game (HCDG) has only two states and two parameters and yet its solution (or, at least, the bulk of it) was the subject of a PhD dissertation~\cite{merz1971homicidal} and a multitude of follow-on publications.
This is, in part, due to the abundance and variety of singularities present in its solution~\cite{basar1982chapter}.
Fortunately, the Lady in the Lake differential game has two states and only one parameter (in its most reduced formulation) and its solution is far simpler than that of the HCDG.
As will be shown, the solution, presented here, concerning the $\min \max$ time to reach the point $E$ contains some singularities of its own.
In particular, the solution contains a Focal Line (FL) -- a line which is, itself, an equilibrium trajectory that has tributary equilibrium trajectories that enter tangentially (c.f., e.g.,~\cite{melikyan2005geometry,breakwell1990simple}).
Additionally, the solution also contains a Universal Line (UL) -- a line which, like the FL, is an equilibrium trajectory, but its tributaries do not enter tangentially.
The UL was introduced in the seminal work by Isaacs~\cite{isaacs1965differential}.

The remainder of this paper is summarized as follows.
\Cref{sec:The_Classical_Lady_in_the_Lake_Scenario} contains a rederivation of the classical Lady in the Lake results.
\Cref{sec:time} presents all of the new results for the $\min \max$ time to reach $E$ differential game.
It's broken down into a subsection on the FL and its tributaries, ~\Cref{sec:Focal_Line}, a subsection on the UL and its tributaries, ~\Cref{sec:Universal_Line}, and a summary of the complete solution.
Lastly, the paper is concluded in~\Cref{sec:Conclusion}.
Regarding notation, many symbols are reused in each section and subsection but are typically defined in a specific way for that context.
For example, the symbol $\H$ is used to denote the Hamiltonian which is defined differently in the classical formulation than it is in the $\min \max$ time formulation.

\section{The Classical Lady in the Lake Scenario\texorpdfstring{~\cite{basar1982chapter}}{}}
\label{sec:The_Classical_Lady_in_the_Lake_Scenario}

In this section, the solution given by Ba{\c s}ar and Olsder in~\cite{basar1982chapter} is rederived in detail for the sake of completeness.
Consider the state space region
\begin{equation*}
    \mathcal{R} = \left\{ (r,θ) \mid 0 \le r \le 1,\ 0\le θ \le π \right\}
\end{equation*}
where $μ < 1$ is the speed of $L$.
Without loss of generality, the angular position of $L$ w.r.t.\ $M$ is assumed to be in the range $θ \in \left[ 0, π \right]$.
The relative dynamics are
\begin{alignat}{2}
    \dot{r} &= μ \cos ψ, \qquad &r(0) &= r_0, \label{eq:dr}\\
    \dot{θ} &= \frac{μ}{r} \sin ψ - ω, \qquad &θ(0) &= θ_0,\ 0 \le t \le t_f, \label{eq:dθ}
\end{alignat}
where $(r_0,\ θ_0) \in \mathcal{R}$, $ψ \in \left[ -\pi, \pi \right]$, and $ω \in \left[ -1, 1 \right]$ (all without loss of generality).
The radius of the lake is set to 1 (again, without loss of generality)\footnote{This reduction of the parameter space to just the ratio of agent speeds, $μ$, can be accomplished through a scaling of space and time.}.
The cost/payoff functional is
\begin{equation}
    \label{eq:classical_cost}
    J\left(r, θ, ψ(\cdot), ω(\cdot)\right) = Φ\left(r_f, θ_f\right) = θ_f,
\end{equation}
which $L$ wishes to maximize and $M$ wishes to minimize.
The terminal surface is given by
\begin{equation}
    \label{eq:classical_terminal_surface}
    ϕ(r, θ) = r - 1 = 0
\end{equation}
The Value function, if it exists, gives the equilibrium cost/payoff of the differential game
\begin{equation}
    \label{eq:V_definition}
    V(r, θ) = \max_{ψ(\cdot)} \min_{ω(\cdot)} θ_f = \min_{ω(\cdot)} \max_{ψ(\cdot)} θ_f.
\end{equation}

We begin by forming the Hamiltonian
\begin{equation}
    \label{eq:H_classical}
    \H = λ_r μ \cos ψ + λ_θ \left( \frac{μ}{r} \sin ψ - ω \right),
\end{equation}
where $λ_r$ and $λ_θ$ are state adjoint variables.
The equilibrium state adjoint dynamics are given by~\cite{bryson1975applied}
\begin{align}
    \dot{λ}_r &= -\frac{\partial \H}{\partial r} = λ_θ \frac{μ}{r^2} \sin ψ \label{eq:dλr_classical}\\
    \dot{λ}_θ &= 0. \label{eq:dλθ_classical}
\end{align}
The last equality implies that $λ_θ(t) = λ_θ \forall t \in \left[ 0, t_f \right]$, i.e., that $λ_θ$ is constant along the entire equilibrium trajectory.
At termination, the state adjoint variables must satisfy~\cite{bryson1975applied}
\begin{align}
    λ_{r_f} &= \frac{\partial Φ}{\partial r_f} + ν \frac{\partial ϕ}{\partial r_f} = ν \\
    λ_θ &= λ_{θ_f} = \frac{\partial Φ}{\partial θ_f} + ν \frac{\partial ϕ}{\partial θ_f} = 1,
\end{align}
where $ν$ is an additional adjoint variable.
The equilibrium heading for $L$ must maximize the Hamiltonian, which implies
\begin{equation}
    \label{eq:ψ_general_classical}
    \cos ψ^* = \frac{λ_r}{\sqrt{λ_r^2 + \frac{1}{r^2}}}, \qquad
    \sin ψ^* = \frac{1}{r \sqrt{λ_r^2 + \frac{1}{r^2}}}.
\end{equation}
Meanwhile, the equilibrium control for $M$ must minimize the Hamiltonian, which implies
\begin{equation}
    \label{eq:ω_general_classical}
    ω^* = 1.
\end{equation}
At termination, the Hamiltonian must satisfy
\begin{equation}
    \label{eq:Hf_classical}
    \H_f = -\frac{\partial Φ}{\partial t_f} - ν \frac{\partial ϕ}{\partial t_f} = 0.
\end{equation}
Furthermore, since the system is time-autonomous and $\frac{\partial \H}{\partial t} = 0$ we have $\H = 0 \forall t$.

Evaluating~\Cref{eq:H_classical} at final time and substituting in the equilibrium controls, \Cref{eq:ψ_general_classical,eq:ω_general_classical}, and solving for $ν$ gives
\begin{equation}
    \label{eq:ν_classical}
    ν = \sqrt{\frac{1}{μ^2} - 1}.
\end{equation}
Note that $ν$ must be positive in order for $\dot{r}_f$ to be positive, which is necessary for $L$ to exit the lake.
Repeating this step for general time gives
\begin{equation}
    \label{eq:λr_classical}
    λ_r = \sqrt{\frac{1}{μ^2} - \frac{1}{r^2}}.
\end{equation}
Again, the negative case of the square root can be ruled out since heading towards the center of the lake is never advantageous along the equilibrium trajectory.
Substituting~\Cref{eq:λr_classical} into~\Cref{eq:ψ_general_classical} gives
\begin{equation}
    \label{eq:ψ_classical}
    \cos ψ^* = \sqrt{1 - \frac{μ^2}{r^2}}, \qquad
    \sin ψ^* = \frac{μ}{r}.
\end{equation}
Since $\sin ψ^* = \frac{μ}{r}$, it must be the case that $r \ge μ$.
That is, the equilibrium control strategy for $L$ is only defined when $r \ge μ$.
As noted in~\cite{basar1982chapter}, $L$'s strategy corresponds to heading away from the tangent of the circle of radius $μ$ and results in a straight line in the non-rotating Cartesian coordinate system.

Substituting the equilibrium control strategies, \Cref{eq:ψ_classical,eq:ω_general_classical}, into the dynamics, \Cref{eq:dr,eq:dθ}, and dividing gives
\begin{align}
    \frac{\mathrm{d}θ}{\mathrm{d}r} &= - \frac{1}{μ} \sqrt{1 - \frac{μ^2}{r^2}} \nonumber\\
    \int_{θ_0}^{θ_f} \mathrm{d}θ &= -\frac{1}{μ} \int_{r_0}^{r_f}\sqrt{1 - \frac{μ^2}{r^2}} \mathop{\mathrm{d}r} \nonumber\\
    θ_f - θ_0 &= -\frac{1}{μ} \left[ \sqrt{r_f^2 - μ^2} - μ \cos^{-1}\left( \frac{μ}{r_f} \right) - \sqrt{r_0^2 - μ^2} + μ \cos^{-1}\left( \frac{μ}{r_0} \right) \right]. \label{eq:flow_classical}
\end{align}
By setting $r_f = 1$ in the above, the Value function is given by
\begin{equation}
    \label{eq:V_classical_solution}
    V(r, θ) = θ - \sqrt{\frac{1}{μ^2} - 1} + \cos^{-1}μ + \sqrt{\frac{r^2}{μ^2} - 1} - \cos^{-1}\left( \frac{μ}{r} \right).
\end{equation}
Define $θ_T = V(μ, π)$, i.e.,
\begin{equation}
    \label{eq:θ_T}
    θ_T = π - \sqrt{\frac{1}{μ^2} - 1} + \cos^{-1}μ
\end{equation}
Note that $L$ can only escape from the point $E = (μ, π)$ if $θ_T > 0$ which implies that $μ > μ_{\text{crit}} \approx 0.21723$.
For the remainder of the paper it is assumed that $L$'s speed is above this critical value.

Now, define the equilibrium trajectory which departs from $E$ and exits the lake as $B$.
Based on~\Cref{eq:flow_classical,eq:θ_T}, then,
\begin{equation}
    \label{eq:B}
    B(r) = π - \sqrt{\frac{r^2}{μ^2} - 1} + \cos^{-1}\left( \frac{μ}{r} \right), \qquad r \in \left[ μ, 1 \right].
\end{equation}
If the state is such that $θ < B(r)$ then $θ_f < θ_T$ from~\Cref{eq:flow_classical,eq:θ_T}.
Therefore, it would be better for $L$ to navigate to the point $E$ and depart along $B$ in order to achieve $θ_f = θ_T$.
\Cref{fig:classical} shows the equilibrium trajectories for the classical solution.
Note the large blank area of the state space for which no unique equilibrium trajectory exists and $L$ is prescribed to swim to the point $E$ and subsequently take the $B$ trajectory.
\begin{figure}[htpb]
    \centering
    \includegraphics[width=0.95\textwidth]{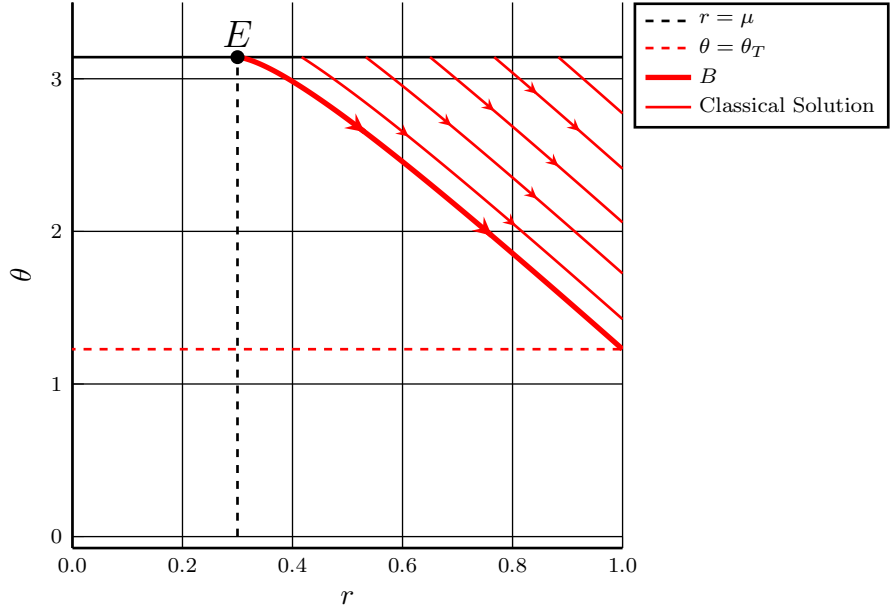}
    \caption{Equilibrium flowfield for the classical Lady in the Lake differential game for $μ = 0.3$.}
    \label{fig:classical}
\end{figure}

The curve $B$ is a barrier surface (in the language of Isaacs~\cite{isaacs1965differential}).
That is, neither agent can steer the state of the system towards or across the surface on their opponent's respective side.
For example, if the state $(r, θ)$, is below $B$, then $L$ cannot force the state onto $B$ (hence why she is prescribed to swim to $E$ first).
Let $\overrightarrow{n}$ be a vector that is normal (pointing up and to the right) to the curve $B$,
\begin{equation}
    \label{eq:n}
    \overrightarrow{n} = \begin{bmatrix} -\frac{\mathrm{d}θ}{\mathrm{d}r} & 1\end{bmatrix}^\top.
\end{equation}
A barrier curve is characterized by
\begin{equation}
    \label{eq:barrier_definition}
    \min_{ω} \max_{ψ} \left\{ \overrightarrow{n} \cdot \begin{bmatrix} \dot{r} & \dot{θ}\end{bmatrix}^\top \right\}= 0.
\end{equation}
Expanding this condition gives
\begin{equation*}
    \min_ω \max_ψ -\frac{\mathrm{d}θ}{\mathrm{d}r} μ \cos ψ + \frac{μ}{r} \sin ψ - ω = 0,
\end{equation*}
which implies that the minimizing and maximizing controls are, respectively, $ω = 1$ and
\begin{equation*}
    \cos ψ = \frac{\frac{\mathrm{d}θ}{\mathrm{d}r}}{\sqrt{\left( \frac{\mathrm{d}θ}{\mathrm{d}r} \right)^2 + \frac{1}{r^2}}}, \qquad
    \sin ψ = \frac{1}{r \sqrt{\left( \frac{\mathrm{d}θ}{\mathrm{d}r} \right)^2 + \frac{1}{r^2}}}.
\end{equation*}
Taking the derivative of~\Cref{eq:B} and substituting into the above expressions shows that~\Cref{eq:barrier_definition} is indeed satisfied.
Furthermore, the condition holds for any curve that is an additive constant w.r.t.\ $B$, hence why there is no hope in $L$ being able to reach $B$ from below.

\section{Min-Max Time to Reach the Antipodal Point}
\label{sec:time}

In this section, we wish to obtain unique trajectories in the region of the state space below the barrier, $B$, that are optimal in some sense.
Specifically, we aim to populate this region with trajectories which reach the point $E = (μ, π)$ such that the time spent getting there is in equilibrium w.r.t.\ the two agents' control strategies.

\subsection{Focal Line}
\label{sec:Focal_Line}

\begin{proposition}
    \label{prop:focal}
    There is a Focal Line (FL) given by
    \begin{equation}
        \label{eq:focal_line}
        \mathcal{F} = \left\{ (r,θ) \mid 0 < r \leq μ,\ θ = π \right\},
    \end{equation}
    wherein $L$'s equilibrium control keeps the state of the state of the system on the line $θ = π$ (i.e., she chooses the heading, $ψ$, s.t.\ $\dot{θ} = 0$):
    \begin{equation}
        \label{eq:focal_line_control}
        \sin ψ_{FL} = \frac{r}{μ},
    \end{equation}
    and $M$'s equilibrium control is
    \begin{equation}
        \label{eq:focal_line_ω}
        ω_{FL} = 1.
    \end{equation}
\end{proposition}
\begin{proof}
    Since $r < μ$, $L$ needs to increase $r \to μ$.
    The goal of $L$ is to reach the point $E = (μ, π)$.
    Any deviation of $θ$ from $π$ will need to be recovered at some point along the trajectory in order to end up at $E$.
    Also $L$'s relative control authority over the $θ$ state is decreasing as she increases $r$.
    Thus any deviation would be best dealt with earlier in the trajectory rather than later.
    Taking this argument to the extreme: it is best for $L$ to keep $θ = π$ along the entire trajectory.
    Regarding $M$'s control, he has some informational advantage in that, technically, $L$ must know his instantaneous control input in order for her to implement her singular control.
    However, if, for example, $M$ were to switch many times (thereby forcing $L$ to have to guess and possibly be wrong many times) $L$ could instead choose $\sin ψ = 0$ and head directly to $E$, arriving in a shorter time.
    In other words, $M$'s efforts to exploit his informational advantage are, themselves, easily exploitable.
    Hence, $M$ should adopt either $ω = 1$ or $ω = -1$ while on the FL, and thus the former is taken without loss of generality.
\end{proof}

\begin{remark}
    The proposed control for $L$ along the FL also keeps the state of the system on the FL itself, which, of course, is one of the properties which makes this surface a FL.
    The other property is that trajectories entering the FL do so tangentially; this property will be proven later.
\end{remark}

Substituting the FL controls, \Cref{eq:focal_line_control,eq:focal_line_ω}, into the dynamics, \Cref{eq:dr}, gives
\begin{align*}
    \dot{r} &= μ \sqrt{1 - \frac{r^2}{μ^2}} \\
            &= \sqrt{μ^2 - r^2}.
\end{align*}
This expression can be used to obtain the amount of time spent on the FL until the point $(r, θ) = (μ, π)$ is reached as follows.


\begin{align*}
    \dot{r} = \frac{\mathrm{d}r}{\mathrm{d}t} &= \sqrt{μ^2 - r^2} \\
    \frac{\mathrm{d}r}{\sqrt{μ^2 - r^2}} &= \mathrm{d}t \\
    \frac{\mathrm{d}r}{μ \sqrt{1 - \frac{r^2}{μ^2}}} &= \mathrm{d}t
\end{align*}
Let $x \equiv \frac{r}{μ}$, and thus, $μ \mathrm{d}x = \mathrm{d}r$:
\begin{align*}
    \frac{\mathrm{d}x}{\sqrt{1 - x^2}} &= \mathrm{d}t \\
\end{align*}
Finally, this equation can be integrated; on the LHS the integration bounds are $x = \tfrac{s}{μ}$ to $1$ (which corresponds to $r$ starting at $s$ and going to $μ$), and the RHS just becomes the time spent on the FL, $t_s$:
\begin{align*}
    \int_{s / μ}^{1} \frac{\mathrm{d}x}{\sqrt{1 - x^2}} &= \int_0^{t_s} \mathrm{d}t = t_s \\
    \implies \left.\sin^{-1}\left( x \right)\right|_{s/μ}^1 &= t_s \\
        \frac{π}{2} - \sin^{-1}\left( \frac{s}{μ} \right) &= t_s
\end{align*}

With the time spent on the FL in hand, the next step is to characterize the FL tributaries, which are those equilibrium trajectories that merge onto the FL.
In order to do so, the game is reformulated as a game which begins from a general initial condition and ends on the FL.

\subsubsection{Equilibrium Heading for FL Tributaries}
\label{sec:Analysis_of_FL_Tributaries}

The terminal manifold is
\begin{equation}
    \label{eq:M}
    \mathcal{M} = \left\{ (r,θ) \mid 0 < r \le μ,\ θ = π \right\}.
\end{equation}
%
$\mathcal{M}$ is also the zero-level set of the function
\begin{equation}
    \label{eq:ϕ_FL}
    ϕ(r, θ) = θ - π.
\end{equation}
For the remainder of the paper, $s$ is used to denote the value of $r$ wherein the state enters the FL.
Thus, for the analysis concerning FL tributaries $r_f = s$ and $θ_f = π$
The terminal cost is the time for $L$ to proceed along the line $θ = π$ from $r = s$ to $r = μ$:
\begin{equation}
    \label{eq:h}
    Φ(s) = \frac{π}{2} - \sin^{-1}\left( \frac{s}{μ} \right).
\end{equation}
%
%
The performance functional is the total time taken by $L$ to reach the point $E \equiv \left( μ, π \right)$ (by way of reaching $(r, θ) = (s, π)$ first)
\begin{equation}
    \label{eq:Js}
    J(ψ(\cdot)) = Φ(s) + \int_0^{t_f} 1 \mathop{\mathrm{d} t},
\end{equation}
which $L$ wishes to minimize.
The Hamiltonian of the system is
\begin{align}
    \H &= λ_r μ \cos ψ + λ_θ \left( \frac{μ}{r} \sin ψ - ω \right) + 1,
    \label{eq:H_FL_generic}
\end{align}
where $λ_r$ and $λ_θ$ are adjoint variables associated with the $r$ and $θ$ states, respectively.
The value of the Hamiltonian at terminal time is given by~\cite{bryson1975applied}
\begin{equation}
    \label{eq:Hf_FL}
    \H_f = -\frac{\partial Φ}{\partial t} - ν \frac{\partial ϕ}{\partial t} = 0.
\end{equation}
Since $\frac{\partial \H}{\partial t} = 0$ and the system's dynamics are time-autonomous we have $\frac{\mathrm{d}\H}{\mathrm{d}t} = 0$ and thus $\H(t) = 0 \forall t \in \left[ 0, t_f \right]$.
The value of the adjoint variables at terminal time are given by~\cite{bryson1975applied}
\begin{align}
    λ_{r_f} &= \frac{\partial Φ}{\partial s} + ν \frac{\partial ϕ}{\partial s} = \frac{-1}{\sqrt{μ^2 - r_f^2}} \label{eq:λrf_FL} \\
    λ_{θ_f} &= \frac{\partial Φ}{\partial θ_f} + ν \frac{\partial ϕ}{\partial θ_f} = ν \label{eq:λθf_FL},
\end{align}
where $ν$ is an additional adjoint variable.
The optimal adjoint dynamics are given by~\cite{bryson1975applied}
\begin{align}
    \dot{λ}_r &= -\frac{\partial \H}{\partial r} = -λ_θ \frac{μ}{r^2}\sin ψ^* \label{eq:λrdot_FL_generic} \\ 
    \dot{λ}_θ &= -\frac{\partial \H}{\partial θ} = 0.
    \label{eq:λθdot_FL}
\end{align}
Since $\dot{λ}_θ = 0$ we have that $λ_θ = ν\ \forall t\in\left[ 0, t_f \right]$.

The equilibrium action for $L$ is one that minimizes the Hamiltonian, $ψ^* = \argmin_ψ \H$, and therefore the vector $\begin{bmatrix} \cos ψ^* & \sin ψ^*\end{bmatrix}$ should be antiparallel with the vector $\begin{bmatrix} λ_r & \frac{ν}{r}\end{bmatrix}$
\begin{equation}
    \label{eq:ψ_FL_generic}
    \cos ψ^* = \frac{-λ_r}{\sqrt{λ_r^2 + \frac{ν^2}{r^2}}}, \qquad
    \sin ψ^* = \frac{-ν}{r \sqrt{λ_r^2 + \frac{ν^2}{r^2}}}
\end{equation}
Similarly, the equilibrium action for $M$ is one that maximizes the Hamiltonian, $ω^* = \argmax_ω \H$, hence,
\begin{equation}
    \label{eq:ω_FL_generic}
    ω^* = -\sign ν.
\end{equation}
%
Assuming $ω > 0$, substituting the equilibrium controls into the Hamiltonian, \Cref{eq:H_FL_generic}, gives
\begin{align}
    \H^* = -μ \sqrt{λ_r^2 + \frac{ν^2}{r^2}} - ν + 1 &= 0 \nonumber \\
    \implies \sqrt{λ_r^2 + \frac{ν^2}{r^2}} &= \frac{1 - ν}{μ}.
    \label{eq:simpler_FL_denom}
\end{align}
Substituting the terminal value of $λ_r$, \Cref{eq:λrf_FL}, and the terminal value of $r$ (i.e., $s$) into the above and solving for $ν$ gives
%
%
%
\begin{equation}
    \label{eq:ν_FL}
    ν = \frac{-s^2}{μ^2 - s^2},
\end{equation}
which is negative since $s < μ$ along the FL.
Substitution into~\Cref{eq:ω_FL_generic} confirms that, indeed, $ω$ is positive.

Substituting~\Cref{eq:ψ_FL_generic} and~\Cref{eq:ν_FL} into~\Cref{eq:H_FL_generic}, evaluating at a general time, and solving for $λ_r$ gives
\begin{equation}
    \label{eq:λr_in_terms_of_r}
    λ_r = \pm \frac{1}{μ^2 - s^2} \sqrt{μ^2 - \frac{s^4}{r^2}}.
\end{equation}
From~\Cref{eq:λrf_FL} it's clear that, at terminal time, $λ_r < 0$ which from~\Cref{eq:ψ_FL_generic} and~\Cref{eq:dr} implies that $\dot{r}_f > 0$.
Also, \Cref{eq:λr_in_terms_of_r} shows that $λ_r = 0$ when $r = \frac{s^2}{μ}$.
Therefore, it must be the case that $λ_r$ (and, consequently, $\dot{r}$) changes sign once the system passes through $r = \tfrac{r_f^2}{μ}$ since $\sign(\dot{λ}_r) = -\sign(\sin ψ^*) = -1$ from~\Cref{eq:λrdot_FL_generic,eq:ψ_FL_generic,eq:ν_FL}.

\begin{lemma}
    \label{lem:ψ_FL}
    The equilibrium heading for $L$ along FL tributaries is given by
    \begin{equation}
        \label{eq:ψ_FL}
        \cos ψ^* = \pm \sqrt{1 - \frac{s^4}{μ^2 r^2}}, \qquad
        \sin ψ^* = \frac{s^2}{μ r}.
    \end{equation}
\end{lemma}
\begin{proof}
    Substitution of~\Cref{eq:simpler_FL_denom,eq:ν_FL} into~\Cref{eq:ψ_FL_generic} gives the above expressions.
\end{proof}
\begin{lemma}
    The equilibrium control for $M$ along FL tributaries is given by
    \begin{equation}
        \label{eq:ω_FL_tributaries}
        ω^* = 1.
    \end{equation}
\end{lemma}
\begin{proof}
    The result follows directly from~\Cref{eq:ω_FL_generic,eq:ν_FL}.
\end{proof}
\begin{lemma}
    \label{lem:FL_tributaries_straight}
    The equilibrium FL tributary trajectory is a straight line in the global Cartesian frame.
\end{lemma}
\begin{proof}
    The proof follows the same steps as the proof for Lemma 2 in~\cite{vonmoll2022circular} and is thus omitted for brevity.
\end{proof}

\begin{lemma}
    \label{lem:enter_FL_tangentially}
    The FL tributaries enter the FL tangentially.
\end{lemma}
\begin{proof}
    Evaluating $\dot{θ}$ from~\Cref{eq:dθ} at terminal time (i.e., $r = s$) and substituting in the equilibrium controls, \Cref{eq:ψ_FL,eq:ω_FL_tributaries}, gives
    \begin{equation*}
        \dot{θ}^*_f = \frac{μ}{s} \left(\frac{s^2}{μ s}\right) - 1 = 0.
    \end{equation*}
    The FL, itself, is a line of constant $θ$, hence the result holds.
\end{proof}

\subsubsection{Equilibrium Flowfield}
\label{sec:Equilibrium_Flowfield}

Now, the equilibrium heading for $L$ can be substituted into $\dot{r}$, \Cref{eq:dr}, to obtain $\dot{r}^*$.
However, the sign of $\cos ψ^*$ is not known directly except at final time (wherein $\cos ψ^*,\ \dot{r} > 0$).
Therefore, it is useful to consider the retrograde equation for $r$ (denoted with a circle instead of a dot, i.e., $\mathring{r} = -\dot{r}$) in order for the initial condition to be fully specified:
\begin{equation}
    \label{eq:ringr_FL}
    \mathring{r}^* = \pm \frac{μ}{r} \sqrt{r^2 - \frac{s^4}{μ^2}}, \qquad r(0) = s,\ \mathring{r}^*(0) < 0.
\end{equation}
As mentioned previously, the sign of $\mathring{r}^*$ is governed by the sign of $λ_r$ which starts (in retrograde time) negative and becomes positive if $r$ reaches the value $\frac{s^2}{μ}$.
Let the retrograde time be denoted by $τ$ such that $τ = 0$ corresponds to $t = t_f$.
Rewriting the above expression,
\begin{align}
    \label{eq:LABEL}
    \frac{\mathrm{d}r}{\mathrm{d}τ} &= \pm \frac{μ}{r} \sqrt{r^2 - \frac{s^4}{μ^2}} \\
    \int_{s}^r \frac{r}{\sqrt{r^2 - \frac{s^4}{μ^2}}} \mathop{\mathrm{d}r} &= \pm \int_0^τ μ \mathop{\mathrm{d}τ} \\
    \left. \sqrt{r^2 - \frac{s^4}{μ^2}}\ \right\rvert_{s}^r &= \pm μ τ \\
    \sqrt{r^2 - \frac{s^4}{μ^2}} - \sqrt{s^2 - \frac{s^4}{μ^2}} &= \pm μ τ \label{eq:r_τ_expression}
\end{align}
\begin{equation}
    \implies r(τ) = \sqrt{s^2 - 2 τ s \sqrt{μ^2 - s^2} + μ^2 τ^2} \label{eq:r_τ}
\end{equation}
Define the time when $r = \frac{s^2}{μ}$ (i.e., when $λ_r$ and $\mathring{r}$ change sign) as $\bar{τ}$; this time is obtained by solving for $τ$ in the negative version of~\Cref{eq:r_τ_expression} with $r = \frac{s^2}{μ}$:
\begin{equation}
    \label{eq:τ_bar}
    \bar{τ} = \frac{s}{μ} \sqrt{1 - \frac{s^2}{μ^2}}.
\end{equation}
Note that this time also corresponds to the time at which $L$ is is closest to the center of the lake along the FL trajectory, i.e., $\min_τ r(τ) = r(\bar{τ})$.

Similarly, for $θ$, after substituting~\Cref{eq:ψ_FL,eq:ω_FL_tributaries} into~\Cref{eq:dθ} and changing to retrograde time we have
\begin{equation}
    \label{eq:θ_ring_FL}
    \mathring{θ}^* = 1 - \frac{s^2}{r^2}, \qquad θ(0) = π.   
\end{equation}
Rewriting the above expression and substituting in $r(τ)$ from~\Cref{eq:r_τ},
\begin{align}
    \frac{\mathrm{d}θ}{\mathrm{d}τ} &= 1 - \frac{s^2}{s^2 - 2 τ s \sqrt{μ^2 - s^2} + μ^2 τ^2} \\
    \int_π^θ \mathrm{d}θ &= \int_0^τ \mathrm{d}τ - s^2 \int_0^τ \frac{1}{s^2 - 2 s τ \sqrt{μ^2 - s^2} + μ^2 τ^2}\mathop{\mathrm{d}τ} \\
    θ - π &= τ - \tan^{-1}\left( \frac{μ^2}{s^2}τ - \sqrt{\frac{μ^2}{s^2} - 1} \right) - \tan^{-1}\left( \sqrt{\frac{μ^2}{s^2} - 1} \right).
\end{align}

The following is stated in order to summarize the results of this section.
\begin{lemma}
    \label{lem:FL_flowfield}  
    The equilibrium flowfield for FL tributaries, parameterized by the entry point on the FL, $s$, is given by
    \begin{equation}
        \label{eq:FL_flowfield}
        \begin{aligned}
            r(τ;s) &= \sqrt{s^2 - 2τ s \sqrt{μ^2 - s^2} + μ^2τ^2}, \\
            θ(τ;s) &= π + τ - \tan^{-1}\left( \frac{μ^2}{s^2}τ - \sqrt{\frac{μ^2}{s^2} - 1} \right) - \tan^{-1}\left( \sqrt{\frac{μ^2}{s^2} - 1} \right).
        \end{aligned}
    \end{equation}
\end{lemma}

\subsubsection{Computation of the FL Entry Point}
\label{sec:Computation_of_the_FL_Entry_Point}

The equilibrium flowfield expressions derived in the previous section are useful for filling a region of the state space with equilibrium trajectories by computing $(r(τ), θ(τ))$ starting from points along the FL.
However, starting (in forward time) from a general position $(r, θ)$, the equilibrium heading of $L$ is unknown as it depends on $s$.
This section describes the process by which $s$ may be computed.

There are two possible cases depending on $L$'s initial condition: 1) $L$'s equilibrium heading has some component of towards the center of the lake and 2) $L$'s equilibrium heading has a component of velocity away from the center of the lake until she reaches the FL.
Consider Case 1.
Let the time of arrival of $L$ to the entry point of the FL, $(s, π)$ be
\begin{equation}
    \label{eq:t_L}
    t_L(s) = \frac{1}{μ} \left( \sqrt{r^2 - \frac{s^4}{μ^2}} + \sqrt{s^2 - \frac{s^4}{μ^2}} \right),
\end{equation}
which is derived based on the fact that $L$'s trajectory is a straight line in the Cartesian frame (per~\Cref{lem:FL_tributaries_straight}) and is tangent to a circle of radius $\frac{r_f^2}{μ}$.
$M$'s time of arrival to the position that is antipodal to $L$ is given by the sum of angles traversed
\begin{equation}
    \label{eq:t_M}
    t_M(s) = θ + \cos^{-1}\left( \frac{s^2}{μ r} \right) + \cos^{-1}\left( \frac{s}{μ} \right) - π.
\end{equation}
Define the function $δ(s) = t_L(s) - t_M(s)$ which is the difference of the agents' respective times of arrival.
The equilibrium entry point onto the FL is thus the smallest possible root of this function, i.e.,
\begin{equation}
    \label{eq:determine_rf}
    s^* = \min s \qquad \text{ s.t. } δ(s) = 0,\ s \in \left( 0,\ μ \right].
\end{equation}
The solution may be obtained numerically as the above expression does not admit a closed-form solution.

Case 2 is similar to Case 1 but with $t_L$ and $t_M$ given, respectively, by
\begin{align}
    t_L(s) &= \frac{1}{μ}\left( \sqrt{s^2 - \frac{s^4}{μ^2}} - \sqrt{r^2 - \frac{s^4}{μ^2}} \right), \label{eq:t_L_case2} \\
    t_M(s) &= θ - \cos^{-1}\left( \frac{s^2}{μ r} \right) + \cos^{-1}\left( \frac{s}{μ} \right) - π .\label{eq:t_M_case2}
\end{align}
In lieu of a more sophisticated method with which to determine whether the initial condition, $(r, θ)$, is in Case 1 or Case 2, the former should be assumed first.
If no solution to~\Cref{eq:determine_rf} can be found, then Case 2 should be assumed.

\subsection{Universal Line}
\label{sec:Universal_Line}

\begin{proposition}
    \label{prop:UL}
    There is a Universal Line (UL) given by
    \begin{equation}
        \label{eq:UL_definition}
        \mathcal{U} = \left\{ (r, θ)~\mid~0 \le r \le 1,\ θ = 0 \right\},
    \end{equation}
    wherein $L$'s equilibrium control strategy is to head directly to the center of the lake and $M$ does not move, i.e.,
    \begin{equation}
        \label{eq:ψ_on_UL}
        \cos ψ_{UL} = -1, \qquad ω_{UL} = 0.
    \end{equation}
\end{proposition}
\begin{proof}
    When $θ = 0$, $M$ has no incentive to move the state of the system to some non-zero $θ$ since doing so increases $L$'s angular separation (which is, ultimately, the thing that $M$ seeks to reduce).
    If $L$ had an angular component of velocity then $θ$ would immediately become non-zero.
    When $θ = 0$, the easiest way for $L$ to drive $θ \to π$ is to pass through the origin.
\end{proof}

Just as in the section on obtaining equilibrium controls for FL tributaries, the game is reformulated as a game which begins from a general initial condition and ends on the UL.

\subsubsection{Equilibrium Heading for UL Tributaries}
\label{sec:Equilibrium_Control_for_UL_Tributaries}

The terminal manifold is the set of states where $θ = 0$, i.e.,
\begin{equation}
    \label{eq:M_UL}
    \mathcal{M} = \left\{ (r, θ) \mid 0 < r \le 1,\ θ = 0 \right\},
\end{equation}
which is also the zero-level set of the function
\begin{equation}
    \label{eq:ϕ_UL}
    ϕ(r, θ) = θ.
\end{equation}
The terminal cost is the time for $L$ to reach the origin along the UL under the proposed UL strategy, \Cref{eq:ψ_on_UL}:
\begin{equation}
    \label{eq:h_UL}
    Φ(r_f, θ_f) = \frac{r_f}{μ}.
\end{equation}
In principle, one may consider the total time to finish out the original game by adding in the time spent along the FL, starting from $(0, π)$ and going to $(μ, π)$, however that is not necessary as that time will be the same for all UL tributaries.
The performance functional is the sum of the time taken to reach the UL and then reach the origin (i.e., \Cref{eq:Js}).
The Hamiltonian is the same as in~\Cref{eq:H_FL_generic}.
Similarly as before, the equilibrium Hamiltonian is zero for all time, and the equilibrium heading is given by~\Cref{eq:ψ_FL_generic} resulting in~\Cref{eq:simpler_FL_denom}.
The terminal adjoint values are
\begin{align}
    λ_{r_f} &= \frac{\partial Φ}{\partial r_f} + ν \frac{\partial ϕ}{\partial r_f} = \frac{1}{μ} \label{eq:λrf_UL} \\
    λ_{θ_f} &= \frac{\partial Φ}{\partial θ_f} + ν \frac{\partial ϕ}{\partial θ_f} = ν. \label{eq:λθf_UL}
\end{align}
Evaluating~\Cref{eq:simpler_FL_denom} at final time results in
\begin{equation}
    \label{eq:ν_UL}
    \sqrt{\frac{1}{μ^2} + \frac{ν^2}{r_f^2}} = \frac{1-ν}{μ}.
\end{equation}
Solving this expression, algebraically, for $ν$ yields $ν = \frac{2 r_f^2}{r_f^2 - μ^2}$ which goes to infinity as $r_f \to μ$; additionally the sign of $ν$ changes depending on whether $r_f \gtrless μ$.
Fortunately, the solution $ν = 0$ is valid for all $r_f, μ \in \left[ 0, 1 \right]$.

\begin{lemma}
    The equilibrium heading for $L$ along UL tributaries is given by
    \begin{equation}
        \label{eq:ψ_UL}
        \cos ψ = -1.
    \end{equation}
\end{lemma}
\begin{proof}
    The result follows from the preceding analysis.
    Ultimately, $L$ must end at the center of the lake and thus the $θ$ state bears no importance while $L$ is \textit{en route}.
    Therefore, the fastest way for $L$ to reach the center of the lake is a straight line path, which is achieved with $ψ = π$.
\end{proof}
Note that, since $ν = 0$, $M$'s control disappears from the Hamiltonian in~\Cref{eq:H_FL_generic} and therefore every value $ω \in \left[ -1, 1 \right]$ is equally optimal.

\subsubsection{Equilibrium Flowfield}
\label{sec:Equilibrium_Flowfield_UL}

In contrast to the FL tributaries, the flowfield for the UL tributaries is simple.
Since $M$'s equilibrium control is undefined on the UL tributaries, we adopt a value of $ω^* = 1$.
\begin{lemma}
    \label{lem:UL_flowfield}
    The equilibrium flowfield for UL tributaries is given by
    \begin{equation}
        \label{eq:UL_flowfield}
        \begin{aligned}
            r(τ) &= μ τ,\qquad &&r(0) \in \left[ 0, 1 \right)\\
            θ(τ) &= τ, \qquad &&θ(0) = 0.
        \end{aligned}
    \end{equation}
\end{lemma}

A direct result of~\Cref{lem:UL_flowfield} is that UL tributaries only exist when $θ \le \frac{r}{μ}$.
The interpretation is that UL tributaries exist when $M$ is close enough to $L$ so as to be able to close their angular separation prior to the latter reaching the center of the lake.

\subsection{Full Solution}
\label{sec:Full_Solution}

The following result pieces together the two types of trajectories covered in the previous section.

\begin{lemma}
    \label{lem:FL_UL_boundary}
    The line segment
    \begin{equation}
        \label{eq:partition}
        \mathcal{P} = \left\{ (r, θ)~\mid~0 \le r \le 1,\ θ = \frac{r}{μ} \right\}
    \end{equation}
    partitions the state space into two regions: one where FL tributaries exist and are optimal and one where UL tributaries exist and are optimal.
    That is, the two regions are mutually exclusive.
\end{lemma}
\begin{proof}
    It was shown previously, in~\Cref{lem:UL_flowfield}, that UL tributaries exist below $\mathcal{P}$.
    The remainder of the proof focuses on showing that FL tributaries exist above $\mathcal{P}$, that is, for $θ > \frac{μ}{r}$.
    Consider the FL tributary for which $r_f \to 0$; this is the most limiting case for FL tributaries as the other endpoint of $\mathcal{F}$ (where $r = μ$) corresponds to already being at the desired point (i.e., the trajectory is the single point $(r, θ) = (μ, π)$).
    From~\Cref{eq:ringr_FL,eq:θ_ring_FL} we have
    \begin{align*}
        \lim_{r_f \to 0} \mathring{r}^*\Big|_{r>0} &= + \frac{μ}{r} \sqrt{r^2 - \frac{0^4}{μ}} = μ \\
        \lim_{r_f \to 0} \mathring{θ}^*\Big|_{r>0} &= 1 - \frac{0^2}{r^2} = 1
    \end{align*}
    These retrograde dynamics result in a line that is parallel to the partition $\mathcal{P}$ and lies arbitrarily close to it since $0 < r_f \ll 1$.
    Two remaining properties are needed in order for the result to hold: 1) that the FL tributaries do not cross one another (and thus no FL tributary crosses below $\mathcal{P}$ as a result of the above analysis), and 2) that the FL tributaries fill the region of the state space above $\mathcal{P}$.
    Both of these properties will be verified, graphically, with an example.
\end{proof}

Based on all of the preceding results of this section, the following theorem summarizes the solution of the min-max time game.

\begin{theorem}
    The solution to the zero-sum differential game of time to reach the antipodal point $E$ is given by the following equilibrium control strategies and associated Value function.
    \begin{align}
        \left( \cos ψ^*,\ \sin ψ^* \right) &= \begin{cases}
            \left( \sqrt{1 - \frac{r^2}{μ^2}},\ \frac{r}{μ} \right) & \text{ if } θ = π, \\
            \left( -1,\ 0 \right) & \text{ if } θ \le \frac{r}{μ}, \\
            \left( \pm \sqrt{1 - \frac{r_f^4}{μ^2r^2}} ,\ \frac{r_f^2}{μr} \right) & \text{ otherwise. }
        \end{cases} \label{eq:ψ_equilibrium} \\
            ω^* &= \begin{cases}
                1 &\text{ if } θ > \frac{r}{μ} \\
                0 &\text{ if } θ = 0 \\
                \text{undef.} & \text { otherwise, }
            \end{cases} \label{eq:ω_equilibrium} \\
                t_f^* &= \begin{cases}
            \frac{π}{2} - \sin^{-1}\left( \frac{r}{μ} \right) & \text{ if } θ = π, \\
            \frac{π}{2} + \frac{r}{μ} & \text{ if } θ \le \frac{r}{μ}, \\
            \frac{π}{2} - \sin^{-1}\left( \frac{s}{μ} \right) + t_L(s), & \text{ otherwise, }
        \end{cases} \label{eq:Value_solution}
    \end{align}
    where $s$ is the solution of~\Cref{eq:determine_rf} and $t_L(s)$ is given by~\Cref{eq:t_L} or~\Cref{eq:t_L_case2} depending on which case applies to the current state as described in~\Cref{sec:Computation_of_the_FL_Entry_Point}.
    Note that the corresponding case determines the sign of $\cos ψ^*$ as well.
\end{theorem}


\begin{remark}
    One may verify that the equilibrium control strategies satisfy the Hamilton-Jacobi-Isaacs (HJI) equation everywhere via direct substitution.
    However, this is true by construction since, in this case, the Hamiltonian is equivalent to the HJI and the control strategies are derived directly from the former.
\end{remark}

\Cref{fig:full} shows the relative state space filled with equilibrium trajectories.
Solutions to the classical game (i.e., the min-max angular separation when $L$ reaches $r=1$) exist above the barrier, $B$.
It is assumed that $L$ would utilize the classical strategy to exit the lake, otherwise, she should swim to $E$ as quickly as possible and then exit the lake along the barrier, $B$.
\begin{figure}[htpb]
    \centering
    \includegraphics[width=0.95\textwidth]{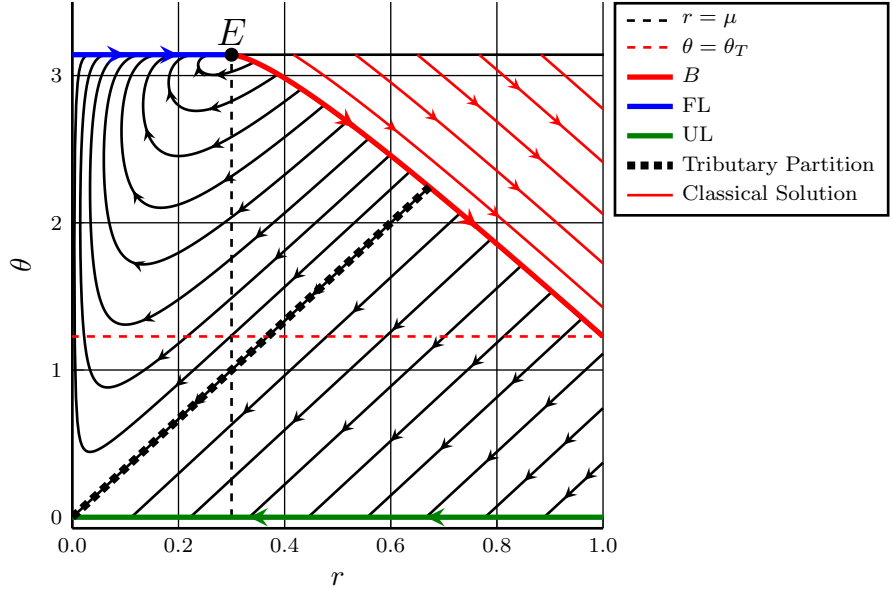}
    \caption{Equilibrium trajectories of the complete Lady in the Lake game with $μ = 0.3$.}
    \label{fig:full}
\end{figure}

\begin{figure}[htpb]
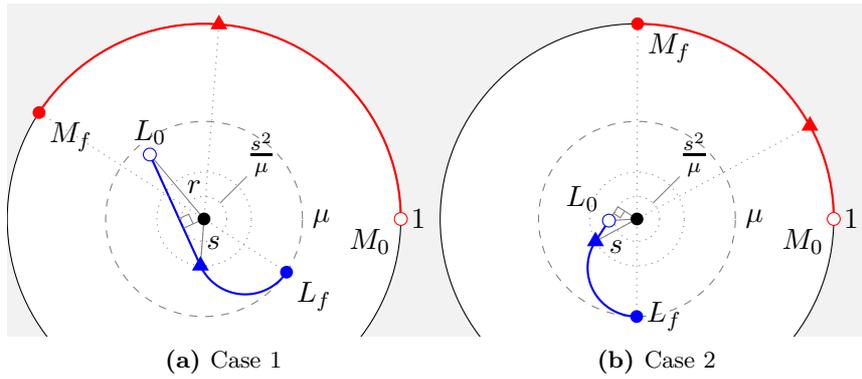

    \centering
    \subfloat[Case 1]{\includegraphics[width=0.45\textwidth]{FL-tributary.tikz}}
    \subfloat[Case 2]{\includegraphics[width=0.45\textwidth]{FL-tributary2.tikz}}
    \caption{
        Focal Line trajectories starting from the tributaries in the non-rotating Cartesian coordinate system.
        In (a), $L$ initially heads towards the tangent of the circle of radius $\frac{s^2}{μ}$ (Case 1), while in (b) $L$ only heads away from the tangent (Case 2).
        Open markers indicate initial positions, triangles designate positions at the moment the FL is reached, and closed markers indicate terminal positions.
    }
    \label{fig:FL-tributary.tikz}
\end{figure}

\section{Conclusion}
\label{sec:Conclusion}

Although the classical Lady in the Lake scenario has been solved for quite some time, the question of what, specifically, to do ``under'' the barrier curve was open.
This paper has addressed that question by providing the $\min \max$ time and associated equilibrium strategies for $L$ to reach the antipodal point.
Subsequent to reaching the antipodal point, $L$ then continues on to reach the shore and obtain the equilibrium terminal angular separation.
Traditional differential game theory methods have been used to obtain the solution of the $\min \max$ time to reach the antipodal point game.
Interestingly, its solution is made up of two singular surfaces and their tributaries.
The approach taken in this paper will serve as a stepping stone to address the more difficult game of $\min \max$ time to escape (i.e., similar to the problem posed in~\cite{mutalik2021math}).

\printbibliography

\end{document}